\documentclass[12pt]{amsart}
\usepackage{amsmath,amsthm,amssymb,amsfonts}
\usepackage[dvipsnames]{xcolor}
\usepackage{tikz}
\usepackage[colorlinks=true,
            linkcolor=violet,
            urlcolor=Aquamarine,
            citecolor=YellowOrange]{hyperref}

\newcommand{\norm}[1]{\left\Vert#1\right\Vert}

\allowdisplaybreaks

\def\R{\mathbb{R}}

\def\n{\mathbf{n}}

\newcommand{\abs}[1]{\left\vert#1\right\vert}
\newcommand{\set}[1]{\left\{#1\right\}}
\newcommand{\pfrac}[2]{\frac{\partial #1}{\partial #2}}

\numberwithin{equation}{section}
\newtheorem{theorem}{Theorem}[section]
\newtheorem{lem}[theorem]{Lemma}
\newtheorem{thm}[theorem]{Theorem}
\newtheorem{pro}[theorem]{Proposition}
\newtheorem{cor}[theorem]{Corollary}
\newtheorem{defi}[theorem]{Definition}

\newcounter{Cnumber}

\makeatletter

\newcommand{\Rmnum}[1]{\expandafter\@slowromancap\romannumeral #1@}
\makeatother
\title[On the Ends of Willmore Surfaces with Curvature Decay]{\bf On the Ends of Willmore Surfaces with Curvature Decay}

\dedicatory{Dedicated to the memory of our teacher, Professor Weiyue Ding.}

\author{Yuxiang Li, Hao Yin}
\address{Department of Mathematical Sciences, Tsinghua University, People's Republic of China}
\email{liyuxiang@tsinghua.edu.cn}

\address{School of Mathematical Sciences, University of Science and Technology of China, People's Republic of China}
\email{haoyin@ustc.edu.cn}

\date{}

\begin{document}

\begin{abstract}
	For a properly embedded Willmore surface $\Sigma$ in $\mathbb R^3$, we prove that if the scale-invariant second fundamental form is sufficiently small near infinity, the surface has finitely many ends. Moreover, if this  scale-invariant quantity vanishes at infinity, or if there is only one end, the total $L^2$-norm of the second fundamental form is finite. %This work establishes a Willmore analog of a classical result for minimal surfaces by Meeks, P\'erez, and Ros.
\end{abstract}

\maketitle

\section{Introduction}

A Willmore surface $\Sigma\subset\mathbb{R}^3$ is a critical point 
of the Willmore functional
\[
W(\Sigma)=\frac14 \int_\Sigma |H|^2\, dV_\Sigma.
\]
It is well-known that minimal surfaces are Willmore.

In minimal surface theory, there is an extensive body of literature describing the global geometry of complete embedded surfaces in $\R^3$. In contrast, the analogous theory for Willmore surfaces is far less developed. Although Willmore surfaces enjoy conformal invariance and share many analytical similarities with minimal surfaces, the fourth-order nature of the Willmore equation makes their global behavior significantly more delicate.

In this paper, we provide an application of the three-circle theorem to the decay estimate of Willmore surfaces—in a way similar to, yet distinct from, the approach in \cite{Li-Yin}. We focus on properly embedded Willmore surfaces in $\R^3$ whose scale‑invariant second fundamental form $\abs{x}\abs{A}$ is sufficiently small near infinity.
Our main results are the following.

%\begin{thm}\label{thm:main1}
%Let $\Sigma\subset\R^3$ be a properly embedded Willmore surface.  Assume that 
%\begin{equation}
%	\label{eqn:main1}
%\lim_{|x|\rightarrow+\infty}|x||A|=0,
%\end{equation}
%where $|x|$ denotes the Euclidean norm of $x$ and $A$ is the second fundamental form of $\Sigma$. Then each end of $\Sigma$
%has finite $L^2$-norm of $A$ and \blue{the total number of ends is bounded}.
%\end{thm}

\begin{thm}\label{thm:main2}
	There exists a universal constant $\epsilon_0>0$ such that for any connected properly embedded Willmore surface $\Sigma\subset\R^3$, if  
\begin{equation}\label{eqn:main2}
\limsup_{|x|\rightarrow+\infty}|x||A|<\epsilon_0,
\end{equation}
then $\Sigma$ has finitely many ends. Moreover, if we further assume that the number of ends is $1$, or
\begin{equation}
	\label{eqn:main1}
\lim_{|x|\rightarrow+\infty}|x||A|=0,
\end{equation}
then $\|A\|_{L^2(\Sigma)}<+\infty$.
\end{thm}

Combining the results in \cite{Chen-Lamm} and \cite{Li}, we obtain the following corollaries.

\begin{cor}
Let $\Sigma\subset\R^3$ be a properly embedded Willmore plane satisfying \eqref{eqn:main2}. Then $\Sigma$ is an affine plane.
\end{cor}

\begin{cor}
Let 
 $\Sigma\subset\R^3$ be a graph over $\R^2$ satisfying \eqref{eqn:main2}. Then 
$\Sigma$ is an affine plane.
\end{cor}

The first part of Theorem \ref{thm:main2} is optimal in the sense that there are examples (due to Kapouleas \cite{ka}) of properly embedded connected minimal surfaces in $\R^3$ with arbitrary (though finitely many) number of planar/catenoidal ends. 

For the second claim in Theorem \ref{thm:main2}, an analogous result was proved in \cite{Meeks-Perez-Ros} for minimal surfaces. More precisely,  without assuming that the number of ends is finite,  Meeks-P\'erez-Ros    showed  that  if $|A|\leq \frac{C}{d(x,x_0)}$ for some $x_0\in\Sigma$, then $\int_\Sigma|A|^2<+\infty$, where $d$ denotes the intrinsic distance on $\Sigma$. It is a natural question whether an analogous statement remains true for Willmore surfaces. In \cite{Li-Yin-Zhou}, for any $c>0$, the authors had an example of a Willmore surface which satisfies
\begin{itemize}
    \item it is properly embedded in $\R^3\setminus \set{0}$;
    \item it has one end at the infinity and a singularity at the origin;
    \item for any $x$ on the surface, $\abs{x}\abs{A(x)}$ is independent of $\abs{x}$;
    \item $\max_{\abs{x}=1} \abs{A(x)}\leq c$.
\end{itemize}
The example implies that the extension of Meeks-P\'erez-Ros's result for Willmore surfaces, if exists, cannot be a local theorem.

A key component in the proof of Theorem \ref{thm:main2} is a three-circle type argument (see Theorem \ref{thm:3circle}). As a first step, we prove that outside a large ball, $\Sigma$ has finitely many connected components and each component when restricted to $B_{2R}\setminus B_R$ (for large $R$) is close to a flat annulus. In the second step, under the additional assumption of only one end, or \eqref{eqn:main1}, we use Theorem \ref{thm:3circle} to show fast decay of $\abs{A}$ which implies the finiteness of $\norm{A}_{L^2(\Sigma)}$. Note that a corollary of Theorem \ref{thm:3circle} is used in the first step (see Proposition \ref{prop:app1}).

The rest of the paper is organized as follows. In Section \ref{sec:pre}, we state a few well-known facts for later use. The proof of Theorem \ref{thm:3circle} and its first application (Proposition \ref{prop:app1}) are in Section \ref{sec:3circle}. In the final section, we prove the main theorem.

{\bf Acknowledgement.} Yuxiang Li is partially supported by National Key R\&D Program of China 2022YFA1005400. Hao Yin is supported by NSFC-12431003.

\section{Preliminaries}
\label{sec:pre}

In this section, we list a few results that are well-known and used in this paper.

\subsection{Graph lemma}

By definition, every surface is the graph of a function in a sufficiently small neighborhood. However, we need quantitative description for our proofs. Hence, we define
\begin{defi}
\label{def:close}
    A piece of surface $\Sigma$ is said to be $\delta$-close to a plane $P$ if the orthogonal projection $\pi$ from $\Sigma$ to $\pi(\Sigma)\subset P$ is a diffeomorphism and there exists a function $u:\pi(\Sigma)\to \R$ satisfying
    \[
    \norm{u}_{C^2(\pi(\Sigma))}\leq \delta
    \]
    such that $\Sigma$ (after a rotation in $\R^3$) is the graph of $u$.
\end{defi}
Note that this definition is not scaling invariant. We also define
\begin{defi}\label{def:closedball}
	A properly embedded surface $\Sigma$ is said to be $\delta$-close to a flat disk of radius $r$ centered at $x\in \Sigma$, if there is a neighborhood $U$ of $x$ in $\Sigma$ and a plane $P$ such that the orthogonal projection onto $P$ is a diffeomorphism from $U$ to $D^P_r$ and $U$ is (up to a rotation and a translation if necessary) the graph of a function $u:D^P_r \to \R$ satisfying 
	\begin{itemize}
		\item $\sup_{D_r}\abs{\nabla u} \leq \delta$;
		\item $\sup_{D_r}\abs{\nabla^{k+1} u}\leq r^{-k}\delta$ for $k\geq 1$.
	\end{itemize}
\end{defi}
Definition \ref{def:closedball} is very close to the so-called $(r,\alpha)$-immersion defined in \cite{Langer}. One obvious difference is that it requires a bound for higher derivatives of $u$ instead of the $C^1$ norm as in \cite{Langer}. This is not a problem, since we will always work with Willmore surfaces and the $\epsilon$-regularity theorem allows us to use any $C^k$ norm. Theorem 2.4 of \cite{Langer} and the $\epsilon$-regularity of Willmore surfaces together imply the following lemma.
\begin{lem}\label{lem:graph}
    Let $\Sigma$ be a properly embedded Willmore surface in $\R^3$ and $x\in \Sigma$. If
    \[
    \sup_{y\in \Sigma\cap B_{r+1}(x)} r\abs{A}(y)\leq \epsilon,
    \]
    then $\Sigma$ is $\delta(\epsilon)$-close to a flat disk of radius $r$ centered at $x$, where
    \[
    \lim_{\epsilon\to 0} \delta(\epsilon)=0.
    \]
\end{lem}
%\cmt{Maybe I need to be explicit that this $\delta(\epsilon)$ depends also on $r$, should be $\delta(\epsilon,r)$. In the proof, we can first scale $B_r(x)$ to $B_1$. The existence of some $\delta$ should be fine. When we scale everything back, the concept of $\delta$-close is not compatible with scaling. Hence, the dependence  of $\delta$ on $r$ would be strange. But any way, there is such $\delta(\epsilon,r)$.}

The following corollary establishes a link between Definition \ref{def:close} and Definition \ref{def:closedball}.
\begin{cor}
	\label{cor:twodef}
For $\delta'>0$ and $r>0$ fixed, if $\Sigma$ is $\delta$-close to a flat disk of radius $r$ centered at $x\in \Sigma$ (as in Definition \ref{def:closedball}) and $\delta$ is smaller than some constant determined by $\delta'$ and $r$, then the piece of surface $\Sigma|_U$ is $\delta'$-close to some plane $P$ (as in Definition \ref{def:close}).
\end{cor}
This is obvious and the proof is omitted.

While Lemma \ref{lem:graph} allows us to control the geometry of a piece of surface with small $\abs{A}$, we also need the following result to combine smaller almost flat pieces together to get a larger one. The proof is very elementary and omitted.
\begin{pro}
	\label{pro:combine}
 For $r_1,r_2>0$ and any $\epsilon>0$, there exists $\delta>0$ such that the following holds. Let $\Sigma_1$ and $\Sigma_2$ be two pieces of surfaces which are $\delta$-close to the planes $P_1$ and $P_2$ respectively. If 
\begin{enumerate}
    \item [(1)] $\Sigma_1\cap B_{r_1}= \Sigma_2\cap B_{r_1}$;
    \item [(2)] the diameters of $\Sigma_1$ and $\Sigma_2$ are bounded by $r_2$,
\end{enumerate} 
 then there is another plane $P$ such that $\Sigma_1\cup \Sigma_2$ is $\epsilon$-close to $P$.
\end{pro}
Note that in the above proposition, $\delta$ depends on $\epsilon$, $r_1$, and $r_2$. Hence, to keep $\delta$ under control, we can apply this proposition only a uniformly bounded number of times, for values of $r_1$ and $r_2$ that are of uniformly comparable sizes.

\subsection{A three-circle lemma for biharmonic function}
Willmore surfces are naturally related to biharmonic functions in the following sense. If the surface is given by the graph of a function $u$, the Willmore equation is 
$$
	\Delta H + \frac{1}{2} H(H^2- 4K)=0
$$
where $H$ and $K$ are expressed in terms $u$ by
\begin{equation}
	\label{eqn:H}
H=\frac{\left(1+u_y^2\right) u_{x x}-2 u_x u_y u_{x y}+\left(1+u_x^2\right) u_{y y}}{\left(1+u_x^2+u_y^2\right)^{3 / 2}}
\end{equation}
and
\begin{equation}
	\label{eqn:K}
K=\frac{u_{x x} u_{y y}-u_{x y}^2}{\left(1+u_x^2+u_y^2\right)^2}.
\end{equation}
The Willmore equation in terms of $u$ is very complicated, however, its linearization at $u=0$ is the biharmonic map equation
$$
	\Delta^2 u =0.
$$
Therefore, if the surface is very close to (a part of) a plane, or equivalently, $u$ is small in reasonable norms, certain behaviors of $u$ is approximately those of a biharmonic function. Also in terms of expansions near $u=0$, the leading term in $H$ is $\Delta u$ and the leading term in $\abs{A}^2$ is $\abs{\nabla^2 u}^2$.

Recall that if $u$ is biharmonic on an annulus of $\mathbb R^2$, there is the well-known expansion (in polar coordinates)
\begin{equation}
	\label{eqn:biexp}
	u(r,\theta) = u_0(r) + \sum_{n=1}^\infty \left[ u_n^c(r) \cos (n\theta) + u_n^s(r) \sin (n\theta) \right]
\end{equation}
where
$$
\begin{aligned}
& u_0(r)=A_0+B_0 \ln r+C_0 r^2+D_0 r^2 \ln r\\
& u_1^{(c)}(r)=A_1 r+B_1 r^{-1}+C_1 r^3+D_1 r \ln r \\
& u_1^{(s)}(r)=\tilde{A}_1 r+\tilde{B}_1 r^{-1}+\tilde{C}_1 r^3+\tilde{D}_1 r \ln r
\end{aligned}
$$
and for $n\geq 2$
$$
\begin{aligned}
& u_n^{(c)}(r)=A_n r^n+B_n r^{-n}+C_n r^{n+2}+D_n r^{2-n} \\
& u_n^{(s)}(r)=\tilde{A}_n r^n+\tilde{B}_n r^{-n}+\tilde{C}_n r^{n+2}+\tilde{D}_n r^{2-n}.
\end{aligned}
$$

\begin{lem}
    \label{lem:lem32} 
    There exists a universal number $R_0>0$. For any $R>R_0$ and any biharmonic function $u$ defined on $D_{R^4}\setminus D_R\subset \mathbb R^2$, if the coefficients $D_1$ and $\tilde{D}_1$ in the expansion of $u$ vanish, then
    \begin{equation}
	    \label{eqn:3circle}
	    \int_{D_{R^3}\setminus D_{R^2}} \abs{\nabla^2 u}^2 dx < \frac{1}{4} \left( \int_{D_{R^4}\setminus D_{R^3}} \abs{\nabla^2 u}^2 dx+ \int_{D_{R^2}\setminus D_R} \abs{\nabla^2 u}^2 dx \right).
    \end{equation}
\end{lem}
Such a property for $u$ is usually known as the three-circle property in the literature. It is very useful in proving decay estimate. In this paper, we shall we use it to prove a decay estimate for the $L^2$ integral of $A$ of a Willmore surface that is a small perturbation of the plane. The proof of Lemma \ref{lem:lem32} is by routine computations and is presented in the appendix for the sake of completeness.

\section{The decay of the second fundamental form} 
\label{sec:3circle}
In this section, we establish a three-circle property for Willmore surfaces (Theorem \ref{thm:3circle}). As an application, we prove that if a properly embedded connected Willmore surface $\Sigma$ satisfies \eqref{eqn:main2} and a component of its restriction to a large ball is very close to a flat disk, then the entire surface is close to a plane (Proposition \ref{prop:app1}). This is useful in the next section in the study of the topology of $\Sigma$ in Theorem \ref{thm:main2}.

\subsection{Three-circle lemma}

We start by making the following assumptions depending on two parameters $\epsilon$ and $R$:
\begin{itemize}
    \item[(A1)] $\Sigma$ is properly embedded in $B_{R^4}\setminus B_R$;
    \item[(A2)] there is a plane $P$ passing through the origin such that $\Sigma \cap (B_{R^4}\setminus B_R)$ is $\epsilon$-close to $P$ (as in Definition \ref{def:close}).
\end{itemize}

For the statement of the main theorem of this part, we need the definition of residues $\tau_1$ and $\tau_2$. They played a key role in \cite{Li-Yin} and \cite{Li-Yin-Zhou}. For the history of their gradual discovery, we refer to the introduction of \cite{Li-Yin}. The definition presented below is the same as in \cite{Li-Yin} except that (1) it is simplified since we are in the codimension one case; (2) we use extrinsic notation instead of a conformal parametrization. 

%\red{I noticed that you used $H=\kappa_1+\kappa_2$. So did we in \cite{Li-Yin}. Let me check.}
\begin{defi}
    Assume that $\Sigma$ intersects $\partial B_r$ transversely for some $r$. For any $c\in \R^3$ and $S\in \mathfrak{so}(\R^3)$, define
    \begin{equation}
        \label{eqn:tau_c}
        \tau_1(\Sigma,c):= \int_{\Sigma\cap \partial B_r} \left( - 2\pfrac{H}{\nu} \n + 2 H \pfrac{\n}{\nu} + \abs{H}^2 \nu \right)\cdot c ds
    \end{equation}
    and
    \begin{equation}
        \label{eqn:tau_S}
        \tau_2(\Sigma,S):= \int_{\Sigma\cap \partial B_r} - 2 (Sx\cdot \n) \pfrac{H}{\nu} + 2 H \pfrac{(Sx\cdot \n)}{\nu} + \abs{H}^2 (Sx\cdot \nu) ds. 
    \end{equation}
    Here $\n$ is the unit normal to the surface $\Sigma$, $\nu$ is the unit normal of $\partial \Sigma$ inside $\Sigma$ and $ds$ is the arc-length element of $\partial\Sigma$.
\end{defi}

It is known that $\tau_1,\tau_2$ are independent of $r$ (see Theorem 2.1 of \cite{Li-Yin}). Moreover, if $\partial B_r\cap \Sigma$ happens to be the boundary of a compact Willmore surface, then $\tau_1$ and $\tau_2$ vanish.

The main result of this section is
\begin{thm}
\label{thm:3circle}
There exist $R_0>0$ and $\epsilon_1>0$ such that for a Willmore surface $\Sigma$ satisfying (A1-A2) with parameter $\epsilon<\epsilon_1$ and $R= R_0+1$ and  
\begin{equation}
    \label{eqn:residue}
\tau_1(\Sigma,c)=0; \qquad \tau_2(\Sigma,S)=0
\end{equation}
for any $c\in \R^3$ and $S\in \mathfrak{so}(\R^3)$, we have
$$
	W_2\leq \frac{1}{4}(W_1+W_3)
$$
where 
\[
W_i:= \int_{\Sigma\cap (B_{R^{i+1}}\setminus B_{R^i})} \abs{A}^2 dV_\Sigma
\]
for $i=1,2,3$.
\end{thm}
\begin{proof}
Let $R_0$ be given in Lemma \ref{lem:lem32} and $R= R_0+1$. Assume that the desired $\epsilon_1$ does not exist, then we may find a sequence of Willmore surfaces $\Sigma_k$ satisfying

\begin{itemize}
    \item[(B1)] $\Sigma_k$ is properly embedded in $B_{R^4}\setminus B_R$;
    \item[(B2)] there is a plane $P_k$ passing through the origin such that $\Sigma_k \cap (B_{R^4}\setminus B_R)$ is $1/k$-close to $P_k$ (as in Definition \ref{def:close});
    \item[(B3)] for any $c\in \R^3$ and $S\in \mathfrak{so}(\R^3)$, we have
    \[
        \tau_1(\Sigma_k,c)=0, \qquad \tau_2(\Sigma_k,S)=0;
    \]
    \item[(B4)] if $W_{k;i}= \int_{\Sigma_k\cap (B_{R^{i+1}}\setminus B_{R^i})} \abs{A_k}^2 dV_{\Sigma_k}$, then
	    \begin{equation}
		    \label{eqn:otherwise}
        W_{k;2} > \frac{1}{4} (W_{k;1} + W_{k;3}).
	    \end{equation}
\end{itemize}

Denote the $\pi$ and $u$ in Definiton \ref{def:close} for $\Sigma_k$ by $\pi_k$ and $u_k$, we have
\begin{equation}
	\label{eqn:ukbound}
	\abs{\nabla u_k}\leq \frac{1}{k}<\frac{1}{10}
\end{equation}
for large $k$. Hence, there is a universal constant $C$ such that
\begin{equation}
	\label{eqn:compare}
	\frac{1}{C}\int_{\pi_k (\Sigma_k\cap (B_{R^{i+1}}\setminus B_{R^i}))} \abs{\nabla^2 u_k}^2 dx\leq W_i(\Sigma_k) \leq C \int_{\pi_k (\Sigma_k\cap (B_{R^{i+1}}\setminus B_{R^i}))} \abs{\nabla^2 u_k}^2 dx.
\end{equation}
Let $\epsilon_k= W_2(\Sigma_k)^{1/2}$ and set
$$
	\tilde{u}_k = \frac{u_k-c_k-a_k x -b_k y}{\epsilon_k}.
$$
where $c_k, a_k,b_k$ are chosen so that the average of $\nabla \tilde u_k$ and $\tilde u_k$ on $D_{2R}\setminus D_R$ is 0.

By Lemma 2.2 and Theorem 2.10 in \cite{Kuwert-Schatzer}, %(It is generally believed that an $L^\infty$ estimate automatically yields higher-order estimates. But in Kuwert’s paper, I found an $L^2$ -norm estimate for the gradient, which is enough.)
we have
\begin{equation}
	\label{eqn:ks}
\|\nabla A_{\Sigma_k}\|_{L^2(B_{R^4-\delta}\setminus B_{R+\delta})}+\| A_{\Sigma_k}\|_{L^\infty(B_{R^4-\delta}\setminus B_{R+\delta})} \leq C \varepsilon_k.
\end{equation}
Note that $(A_{\Sigma_k})_{ij}=\frac{(u_k)_{ij}}{\sqrt{1+|\nabla u_k|^2}}\n$. Since the $\nabla$ in \eqref{eqn:ks} is the covariant derivative of the normal bundle, we have
$$
(\nabla_l A_{\Sigma_k})_{ij}=\left(\frac{(u_k)_{ijl}}{\sqrt{1+|\nabla u_k|^2}}+O(|\nabla^2 u_k|^2)\right)\n.
$$
Due to \eqref{eqn:ukbound}, we obtain
$$
\|\nabla^3 u_k\|_{L^2(\pi_k(\Sigma_k \cap (B_{R^4-\delta}\setminus B_{R+\delta})))}<C\varepsilon_k,
$$
%If $\varepsilon_1$ is small, due to the control over $\nabla u_k$, the well-known $\varepsilon$-regularity implies
%$$
%	\norm{\nabla^l A_{\Sigma_k}}_{L^\infty(\Sigma_k \cap (B_{R^4-\delta}\setminus B_{R+\delta}))} \leq C(l) \epsilon_k.
%$$
%This control over the derivatives of $A_{\Sigma_k}$ passes on to $u_k$. 
%$$
%	\norm{\nabla^l u_k}_{L^\infty(\pi_k(\Sigma_k \cap (B_{R^4-\delta}\setminus B_{R+\delta})))} \leq C(l) \epsilon_k,
%$$
which implies that
$$
	\norm{\nabla^3 \tilde u_k}_{L^2(\pi_k(\Sigma_k \cap (B_{R^4-\delta}\setminus B_{R+\delta})))} \leq C.
$$
Therefore, we may assume that the sequence $\tilde{u}_k$ converges locally strongly (in $W^{3,2}$) in $D_{R^4}\setminus D_R\subset \mathbb R^2$ to a biharmonic function $\tilde{u}_\infty$. By \eqref{eqn:compare} and the definition of $\epsilon_k$,
$$
	\int_{D_{R^3}\setminus D_{R^2}} \abs{\nabla^2 \tilde{u}_k}^2 dx \geq c>0,
$$
hence $\tilde{u}_\infty$ is nontrivial.
%\cmt{Here and elsewhere in the paper, we need to choose a better notation for balls on $R^2$ or some plane.}

Due to the explicit formula for $A_{\Sigma_k}$ in terms $u_k$
$$
	|A|^2=\frac{\left[\left(1+u_y^2\right) u_{x x}-2 u_x u_y u_{x y}+\left(1+u_x^2\right) u_{y y}\right]^2}{\left(1+|\nabla u|^2\right)^3}-2 \frac{u_{x x} u_{y y}-u_{x y}^2}{\left(1+|\nabla u|^2\right)^2} .
$$
we obtain
$$
	\lim_{k\to \infty} \frac{1}{\epsilon_k^2} \abs{A_k}^2 = \abs{\nabla^2 \tilde{u}_\infty}^2.
$$
Therefore, \eqref{eqn:otherwise} implies
$$
\int_{D_{R^3}\setminus D_{R^2}} \abs{\nabla^2 \tilde{u}_\infty}^2 dx \geq \frac{1}{4} \left(  \int_{D_{R^2}\setminus D_R} + \int_{D_{R^4}\setminus D_{R^3}}\right)	\abs{\nabla^2 \tilde{u}_\infty}^2 dx.
$$
To derive a contradiction from Lemma \ref{lem:lem32} and complete the proof of Theorem \ref{thm:3circle}, it suffices to show that in the expansion of $\tilde{u}_\infty$, the coefficients $D_1$ and $\tilde{D}_1$ are both zero.
For that purpose, we recall
    \begin{equation*}
        \tau_2(\Sigma,S):= \int_{\Sigma\cap \partial B_r} - 2 (Sx\cdot \n) \pfrac{H}{\nu} + 2 H \pfrac{(Sx\cdot \n)}{\nu} + \abs{H}^2 (Sx\cdot \nu) ds. 
    \end{equation*}
%    There is explicit formula for $H_k$ in terms of $u_k$:
%    $$
%    	H=\frac{\left(1+u_y^2\right) u_{x x}-2 u_x u_y u_{x y}+\left(1+u_x^2\right) u_{y y}}{2\left(1+|\nabla u|^2\right)^{3 / 2}}, \quad K=\frac{u_{x x} u_{y y}-u_{x y}^2}{\left(1+|\nabla u|^2\right)^2} .
%    $$
For fixed $r\in (R,R^2)$, the assumption that $\tau_2(\Sigma_k,S)=0$ and the equations \eqref{eqn:H} and \eqref{eqn:K} imply that if we take the limit $k\to \infty$, and set $\nu= (\cos \theta, \sin\theta, 0)$ and $\n = (0,0,1)$, we get
\[
\int_0^{2\pi} -(S( r\cos\theta, r\sin\theta,0) \cdot (0,0,1)) \partial_r (\triangle \tilde{u}_\infty) + \triangle \tilde{u}_\infty  (S(\cos\theta, \sin\theta,0) \cdot (0,0,1)) d\theta =0.
\]

In particular, by taking
\[
S=
\left(
\begin{array}{ccc}
   0  & 0 & 1 \\
   0  & 0 & 0 \\
   -1  & 0 & 0 \\
\end{array}
\right)
\quad \text{or} \quad
\left(
\begin{array}{ccc}
   0  & 0 & 0 \\
   0  & 0 & 1 \\
   0  & -1 & 0 \\
\end{array}
\right),
\]
we find that
\begin{equation}
    \label{eqn:nobad}
0= \int_0^{2\pi} \cos\theta (-r\partial_r (\Delta \tilde u_\infty)+ \Delta \tilde u_\infty) d\theta = \int_0^{2\pi} \sin \theta (-r\partial_r (\Delta \tilde u_\infty)+ \Delta \tilde u_\infty) d\theta.
\end{equation}
By \eqref{eqn:biexp}, this is equivalent to $D_1=\tilde{D}_1=0$.
%With \eqref{eqn:nobad}, Lemma \ref{lem:lem32} implies that
%\begin{equation}
%    \label{eqn:hf}
%\int_{D_{R^3}\setminus D_{R^2}} \abs{\phi_\infty}^2 dx \geq e^{-q(L_0+1)} \left( \int_{D_{R^4}\setminus D_{R^3}} \abs{\phi_\infty}^2 dx +\int_{D_{R^2}\setminus D_R} \abs{\phi_\infty}^2 dx
%\right).
%\end{equation}
%In fact, by taking $\phi_\infty$ as a harmonic function in cylindrical coordinates $(t,\theta)$ where $t=\log r$, there is a well-known expansion of $\phi_\infty$ of the form:
%\[
%\phi_\infty(t,\theta) = a+ bt + \sum_{k=1}^\infty \left( (a_k e^{-kt} + b_k e^{kt} ) \cos \theta + (a'_k e^{-kt} + b'_k e^{kt})\sin \theta \right).
%\]
%The property \eqref{eqn:nobad} is equivalent to
%\[
%a_1=a_1'=0
%\]
%in the above expansion. Hence, Lemma \ref{lem:lem32} implies \eqref{eqn:hf}. But, this is a contradiction to (B4) since $\phi_k$ converges locally in $C^1$ norm to $\phi_\infty$ on $B_{R^4}\setminus B_R$. This completes the proof of Theorem \ref{thm:3circle}.
\end{proof}

% Note: $\epsilon_1$ is the 3-circle threshold; it is used in $\epsilon$-close to a plane.

\subsection{An application}

\begin{pro}
\label{prop:app1}
For any $\delta>0$, there exists some $\epsilon_2>0$ such that the following holds.
    Let $\Sigma$ be a properly embedded connected Willmore surface satisfying
    \begin{equation}
        \label{eqn:app1}
    \sup_{\R^3 \setminus B_1} \abs{x} \abs{A}\leq \epsilon_2.
    \end{equation}
   If for some plane $P$ passing through the origin, a component (denoted by $\Sigma_0$) of $\Sigma\cap B_2$ is $\epsilon_2$-close to $P$, then there exists a plane $P'$ such that $\Sigma$ is $\delta$-close to $P'$.
\end{pro}

\begin{proof}
   For any $r>2$, consider $\Sigma\cap B_r$. We do not know if it is connected for now. We then denote the component containing $\Sigma_0$ by $\Sigma_0(r)$. 

   Let $R=R_0+1$ for $R_0$ given in Theorem \ref{thm:3circle}. By Lemma \ref{lem:graph} and Proposition \ref{pro:combine}, if $\epsilon_2$ is small, the assumption that $\Sigma_0(2)$ is $\epsilon_2$-close to $P$ implies that $\Sigma_0(R^5)$ is $\delta$-close to some plane $P'$. In fact, by \eqref{eqn:app1} and Lemma \ref{lem:graph}, for each $y\in \Sigma\cap (B_{R^5}\setminus B_2)$, $\Sigma$ is $\delta(\epsilon_2)$-close to a flat disk of radius $\abs{y}/4$ centered at $y$ (see Definition \ref{def:closedball}). Here $\delta(\epsilon_2)$ is a function of $\epsilon_2$ satisfying $\lim_{\epsilon_2\to 0} \delta(\epsilon_2)=0$. We can choose finitely many such pieces and use Proposition \ref{pro:combine} to combine them together if $\epsilon_2$ is small (depending on $R_0$). 

   The number $5$ in $\Sigma_0(R^5)$ is somewhat arbitrary in the sense that the above argument works for any finite number $k$ as long as we allow $\epsilon_2$ to depend on $k$. Thanks to Theorem \ref{thm:3circle}, we shall prove the claim: 

   \textbf{Claim.} There exists a small constant $\epsilon_2>0$ that is independent of $k$ ($k\geq 5$) such that if $\Sigma_0(R^k)$ is $\delta$-close to some plane $P$, then $\Sigma_0(R^{k+1})$ is $\delta$-close to (maybe) another plane $P'$.

Assuming the claim and since $\delta$ is independent of $k$, we can do induction on $k$ and take the limit to see that $\Sigma_0(\infty)$ is $\delta$-close to some plane $P'$. Since $\Sigma$ is connected by assumption, we know 
\[
\Sigma = \Sigma_0(\infty),
\]
which completes the proof of Proposition \ref{prop:app1}.

It remains to prove the claim. The statement in the claim is not scaling invariant. We shall prove it by establishing some exponential decay for $A_{\Sigma_0}$ in the annulus $\Sigma_0(R^{k+1})\setminus \Sigma_0(R)$.

By the assumption of the claim, $\partial \Sigma_0(R^l)$ is $\delta$-close to a circle of radius $R^l$. By Lemma \ref{lem:graph}, for any $y\in \Sigma_0(B_{R^{k+1}}\setminus B_{R^2})$, $\Sigma$ is $\delta(\epsilon_2)$-close to a flat disk of radius $\abs{y}/4$ centered at $y$. Hence, we know $\Sigma_0\cap (B_{R^{l+1}}\setminus B_{R^{l-2}})$ is topologically an annulus. Note that being $\delta(\epsilon_2)$-close to a flat disk as defined in Definition \ref{def:closedball} is scaling invariant. If $\psi_l(y)=\frac{y}{R^{l-3}}$, by choosing $\epsilon_2$ small and applying Proposition \ref{pro:combine}, we find that for any $l=4,5,\dots, k$
$$
	\psi_l \left( \Sigma_0\cap (B_{R^{l+1}}\setminus B_{R^{l-2}}) \right)(:=\Sigma_0^{(l)})
$$
satisfies the assumption (A2) in Theorem \ref{thm:3circle}. Since $\Sigma_0(r)$ (for any $r\leq R^{k+1})$ is topologically a disk, the assumption on $\tau_1$ and $\tau_2$ holds automatically. Applying Theorem \ref{thm:3circle} to $\Sigma_0^{(l)}$ for each $l=4,5,\dots,k$, implies that
$$
	\int_{\Sigma_0\cap (B_{R^{l}}\setminus B_{R^{l-1}})} \abs{A_{\Sigma_0}}^2 dV_{\Sigma_0} \leq \frac{1}{4} \left(  \int_{\Sigma_0\cap (B_{R^{l+1}}\setminus B_{R^{l}})} \abs{A_{\Sigma_0}}^2 dV_{\Sigma_0}+ \int_{\Sigma_0\cap (B_{R^{l-1}}\setminus B_{R^{l-2}})} \abs{A_{\Sigma_0}}^2 dV_{\Sigma_0}\right).
$$
%We need to improve this $2\delta$ to $\delta$. For that purpose, we need Theorem \ref{thm:3circle}. Without loss of generality, we may assume that $2\delta$ is smaller than the $\epsilon_1$ in Theorem \ref{thm:3circle} so that (A2) holds for $\Sigma_0(R^{i+4})\setminus \Sigma_0(R^{i+1})$ ($i=1,2,\cdots ,k-4$). Note that by our assumption, $\Sigma_0(R^k)\cap \partial B_r$ bounds a topological disk for $r<R^k$, which implies the vanishing of $\tau_1$ and $\tau_2$ when we apply Theorem \ref{thm:3circle} to $\Sigma_0(R^{i+4}) \setminus \Sigma_0(R^{i+1})$ for $i=1,2,\cdots ,k-4$. Finally, Lemma \ref{lem:graph} and \eqref{eqn:app1} imply that
%$$
%\int_{\Sigma_0\cap (B_{R^i}\setminus B_{R^{i-1}})} \abs{H_{\Sigma_0}}^2 dV_{\Sigma_0} \leq C\epsilon_2.
%$$
The above three-circle property implies the exponential decay estimate
\begin{equation*}
\int_{\Sigma_0\cap (B_{R^l}\setminus B_{R^{l-1}})} \abs{A_{\Sigma_0}}^2 dV_{\Sigma_0} \leq C \epsilon_2 e^{- \alpha \min(l,k-l)}
\end{equation*}
for some $\alpha>0$. The $\epsilon$-regularity of Willmore surface then yields the pointwise estimate (for some $\alpha'>0$)
\[
\abs{x}\abs{A_{\Sigma_0}}(x) \leq C \epsilon_2 \left( \abs{x}^{-\alpha'} + (R^{k+1}/\abs{x})^{-\alpha'} \right), \qquad \forall x\in \Sigma_0(R^k)\setminus \Sigma_0(R).
\]
This pointwise decay for $A_{\Sigma_0}$ implies that if $\epsilon_2$ is small (depending on $\delta$), then $\Sigma_0(R^{k+1})$ is $\delta$-close to $P'$ where $P'$ is the tangent plane of $\Sigma_0(R^{k+1})$ at some $x\in \Sigma_0(R^k)\cap \partial B_{R^{k/2}}$. Hence the proof of the claim is done.
\end{proof}

% $\epsilon_2$ is used to say a big disk grows forever to be a plane. It is used to bound $\abs{x}\abs{A}$ and $\epsilon_2$-close to a disk of radius 2.

\section{Geometry and topology of ends}
\label{sec:proof}

The goal of this section is to prove Theorem  \ref{thm:main2}. 
A first step of the proof is to understand the basic topology and shape of the surface outside $B_1$ satisfying
\begin{equation}
    \label{eqn:goodinf}
\abs{x}\abs{A}(x)\leq \epsilon, \quad \forall x\in \Sigma \cap (\R^3\setminus B_1).
\end{equation}

\subsection{Estimate on the normal direction of the surface}
For any point far away from the center, our first result shows its tangent plane almost passes the origin under some extra assumptions.

\begin{pro}
\label{prop:trans}
For any $\delta>0$, there exist $\epsilon_3>0$ and $K>0$ such that the following holds. Let $\Sigma$ be a properly embedded Willmore surface in $\R^3 \setminus B_1$  satisfying \eqref{eqn:goodinf} with $\epsilon<\epsilon_3$.
    Assume that $\Sigma$ is connected and $\Sigma \cap \partial B_1\ne \emptyset$.
   Then for any $x\in \Sigma\cap (\R^3 \setminus B_{K})$, if $\n_x$ is the unit normal vector of $\Sigma$ at $x$, then
   \begin{equation}
       \label{eqn:trans}
       \abs{\langle \n_x, \frac{x}{\abs{x}}\rangle} \leq \delta.
   \end{equation}
\end{pro}
\begin{proof}
    For given $\delta\in (0,1)$, let $K$ be the unique positive solution to 
    $$
    \delta= \sqrt{1-\frac{4}{K^2}}.
    $$
    Without loss of generality, we assume that $\delta$ is small and $K\geq 5$. By elementary geometry, the fomula implies that for $\abs{x}\geq K$, if $P$ is a plane passing $x$ and the normal vector $\n_P$ satisfies
$$
	\abs{\langle \n_P, \frac{x}{\abs{x}} \rangle}>\delta,
$$
then the distance from the origin to $P$ is larger than $2$.

    For $x\in \Sigma\cap (\R^3 \setminus B_{K})$, let $\lambda=\frac{\abs{x}}{K}$ and consider the scaling $\psi(x)=\frac{x}{\lambda}$ in $\mathbb R^3$. Since $\Sigma$ is connected, we observe that $\psi(\Sigma)$ also satisfies the same assumptions and if the proposition holds for $\psi(\Sigma)$, it holds for $\Sigma$. Hence, we may assume without loss of generality that $\abs{x}=K$ in the rest of the proof.

    Prove by contradiction. If \eqref{eqn:trans} is not true for $x\in \Sigma\cap \partial B_K$, we shall get a contradiction by requiring $\epsilon_3$ small. Let $P$ be the tangent plane of $\Sigma$ at $x$. Since $\abs{\langle \n_x, \frac{x}{\abs{x}}\rangle} >\delta$, by the definiton of $K$,  the minimal distance between the origin and $P$ is larger than $2$.

    We proceed to find a piece of $\Sigma$ that is $\delta$-close to $P$. In the following induction which ends in finitely many steps (depending on $K$), we will keep asking $\epsilon_3$ to be smaller. We start with the intersection of $\Sigma$ with $B_{\frac{1}{5}\abs{x}}(x)$. Let $\Sigma_0(x)$ be the connected component of $\Sigma\cap B_{\frac{1}{5}\abs{x}}(x)$ containing $x$. Since $B_{\frac{1}{5}\abs{x}}(x)$ is disjoint with $B_1$, by \eqref{eqn:goodinf}, Lemma \ref{lem:graph} and Corollary \ref{cor:twodef}, $\Sigma_0(x)$ is $\delta(\epsilon_3)$-close to $P$. Next, we pick $y\in \Sigma_0(x)$ such that (1) $\abs{x-y}> \frac{1}{10}\abs{x}$; (2) the intersection of $B_{\frac{1}{5}\abs{x}}(x)$ and $B_{\frac{1}{5}\abs{y}}(y)$ contains a ball of radius $\frac{1}{20}\abs{x}$. Similarly, the component of $\Sigma\cap B_{\frac{1}{5}\abs{y}}(y)$ containing $y$ is also $\delta(\epsilon_3)$-close to $P$. By Proposition \ref{pro:combine}, we get a larger piece of $\Sigma$, denoted still by $\Sigma_0(x)$, which is $\delta(\epsilon_3)$-close to $P$. We repeat the above process, until $\Sigma_0(x)$ is properly embedded in $B_K$. The induction stops in finite steps because the radius of the balls is bounded from below by $1/5$ of the distance between the center and the origin. At the end of the induction, $\Sigma_0(x)\cap B_K$ is  close to the intersection of a plane $P$ with $B_K$ and the distance between $P$ and the origin is no less than $2-\delta$.

%   Let $x_0$ be the projection of the origin onto the plane $P$. Consider the ball $B$ (in $P$) centered at $x_0$ with radius $\abs{x-x_0}$. There exist finitely many points $\set{y_i}$ in $B$ and radius $r_i$ such that 
%   \begin{itemize}
%       \item $B_{r_i}(y_i)\cap B_1 =\emptyset$;
%       \item $B$ is covered by $\set{B_{r_i}(y_i)}$;
%       \item the total number of $\set{y_i}$ is bounded by a universal number.
%   \end{itemize}
%   It is important that the total number of balls is independent of $\abs{x}$. The idea is that we can use larger balls for $y_i$ that is further away from the origin.  For each $B_{r_i}(y_i)$, by choosing $\epsilon_3$ small, the intersection of $\Sigma$ with $B_{r_i}(y_i)$ is $\delta$-close to the plane $P$. Since the total number is bounded, we may conclude that there is a piece of $\Sigma$ in $B_{\abs{x-x_0}}(x_0)$ that is $\delta$-close to $P$. Obviously, we have
    If we shift the origin to the projection of the origin on $P$, denoted by $x_0$, the plane $P$ passes the new origin and \eqref{eqn:goodinf} implies that
   \[
   \abs{x-x_0} \abs{A}(x)\leq 2\epsilon_3, \quad \forall x \in \R^3 \setminus B_{K/2}(x_0).
   \]
   Now, Proposition \ref{prop:app1} implies that as long as $\epsilon_3$ is sufficiently small, $\Sigma$ is $\delta$-close to some plane $P'$. Since $P$ and $P'$ are $2\delta$-close inside $B_K$ and the minimal distance from the origin to $P$ is larger than $3/2$, this is a contradiction that $\Sigma$ has nontrivial intersection with $\partial B_1$.
\end{proof}

\subsection{The intersection between $\Sigma$ and a large sphere}

With Proposition \ref{prop:trans}, for $R_1>K+1$ (with $K$ given by Proposition \ref{prop:trans}), $\Sigma$ intersects with $\partial B_{R_1}$ transversely. The intersection is a finite union of closed curves on $\partial B_{R_1}$. Let $\gamma$ be one of the components of the intersection. The goal of this part is to study the geometry of $\gamma$.

\begin{lem}
	\label{lem:42}
For any $\delta'>0$, there exist $\epsilon_4>0$ and $K>0$ such that if $\Sigma$ is a properly embedded Willmore surface in $\R^3 \setminus B_1$ satisfying \eqref{eqn:goodinf} with $\epsilon<\epsilon_4$ and $\gamma$ is \textbf{one component} of the intersection of $\Sigma$ and $\partial B_{R_1}$ (with $R_1>K+1$), then the following holds. For any two points $x_1,x_2$ on $\gamma$ with intrinsic distance $d_\gamma(x_1,x_2)\leq 2\pi R_1$, the distance between $x_2$ and the great circle on $\partial B_{R_1}$ determined by the tangent vector of $\gamma$ at $x_1$ is smaller than $\delta' R_1$.
\end{lem}

\begin{proof}
	$K$ is to be determined at the end of the proof.
Assume $\gamma$ is parametrized by arc-length and $x_1=\gamma(0)$. Since $\gamma$ is the intersection of $\Sigma$ and the sphere $\partial B_{R_1}$, we have
$$
	\gamma'(s)= \n_{\gamma(s)} \times \frac{\gamma(s)}{\abs{\gamma(s)}}.
$$
Let $\gamma''(s)$ be the covariant derivative of $\gamma'(s)$ as a vector field of $\partial B_{R_1}$ along $\gamma$. Then
$$
\gamma''(s)= \pi_{\gamma(s)}\left[ \left( \n_{\gamma(s)} \right)' \times \frac{\gamma(s)}{\abs{\gamma(s)}} + \n_{\gamma(s)} \times \left( \frac{\gamma(s)}{\abs{\gamma(s)}} \right)'\right],
$$
where $\pi_{\gamma(s)}$ is the projection onto the tangent plane of $\partial B_{R_1}$	 at $\gamma(s)$.
For the first term in the right-hand side above, we estimate, using the fact that $\abs{\gamma(s)}\equiv R_1$ and \eqref{eqn:goodinf}
$$
	\abs{\pi_{\gamma(s)}\left[ \left( \n_{\gamma(s)} \right)' \times \frac{\gamma(s)}{\abs{\gamma(s)}} \right]} \leq \abs{\n'_{\gamma(s)}} \leq CR_1^{-1}\epsilon.
$$
For the second term,  we have
$$
\pi_{\gamma(s)}\left[ \n_{\gamma(s)} \times \left( \frac{\gamma(s)}{\abs{\gamma(s)}} \right)' \right] = R^{-1}_1 (\n_{\gamma(s)})^\perp \times  \gamma'(s).
$$
Here $(\n_{\gamma(s)})^{\perp}$ is the projection of $\n_{\gamma(s)}$ onto the normal space of $\partial B_{R_1}$ at $\gamma(s)$. By \eqref{eqn:trans}, $\abs{\n_{\gamma(s)}\cdot \frac{\gamma(s)}{\abs{\gamma(s)}}}\leq \delta(\epsilon_4,K)$, where $\delta(\epsilon_4,K)$ is some function satisfying
$$
	\lim_{\epsilon_4\to 0; K\to \infty} \delta(\epsilon_4,K)=0.
$$
Hence,
$$
\abs{ \pi_{\gamma(s)} \left[  \n_{\gamma(s)} \times \left( \frac{\gamma(s)}{\abs{\gamma(s)}} \right)'  \right]} \leq C\delta(\epsilon_4,K) R_1^{-1}.
$$
In summary, we have obtained that
$$
\abs{\gamma''(s)}\leq C \delta(\epsilon_4,K) R_1^{-1}.
$$
The remaining part of the proof is a stability argument for the ODE system satisfied by a geodesics. More precisely, if $\tilde{\gamma}$ is the geodesic (great circle) parametrized so that $\tilde{\gamma}(0)=x_1$ and $\tilde{\gamma}'(0)=\gamma'(0)$. Then for $\abs{s}\leq 2\pi R_1$, 
$$
	d_{\partial B_{R_1}}(\gamma(s),\tilde{\gamma}(s))\leq C \delta(\epsilon_4,K) R_1^{-1} \times R_1^2 \leq \delta' R_1
$$
if we choose $\epsilon_4$ small and $K$ large.
\end{proof}

As a corollary, 
\begin{cor} \label{cor:circle}
    For any $\delta''>0$, there exist $\epsilon_5>0$ and $K>0$ such that if $\Sigma$ is a properly embedded Willmore surface in $\R^3 \setminus B_1$ satisfying \eqref{eqn:goodinf} with $\epsilon<\epsilon_5$ and $\gamma$ is one component of $\Sigma\cap \partial B_{R_1}$ (with $R_1>K+1$), then $\gamma$ is in the $\delta'' R_1$-neighborhood of some great circle on $\partial B_{R_1}$.
\end{cor}
\begin{proof}
First, pick any $x_1\in \gamma$. Let $P$ be the plane containing the origin, $x_1$ and the vector $\gamma'(x_1)$. Denote by $\gamma_1$ the part of $\gamma$ consisting of points whose intrinsic distance to $x_1$ is smaller than $\pi (R_1+1)$. Let $Q$ be the subset of points in $\partial B_{R_1}$ whose distance to $P$ is smaller than $R_1/5$.

We first claim that the image of $\gamma$ is a subset of $Q$. If otherwise, then there is $x_2\in \gamma$ with $x_2\notin Q$. Denote by $\gamma_2$ the part of $\gamma$ consisting of points whose intrinsic distance to $x_2$ is smaller than $\pi (R_1+1)$. As illustrated in Figure \ref{fig:1}, as long as $\epsilon_5$ is small and $K$ is large, Lemma \ref{lem:42} implies that $\gamma_1$ and $\gamma_2$ are each very close to a distinct great circle. Since these great circles intersect, $\gamma_1$ and $\gamma_2$ mush also intersect, contradicting the fact that $\gamma$ is an embedded curve on $\partial B_{R_1}$ without self-intersections.
\begin{figure}[htb]
    \centering
    \includegraphics[width=0.4\linewidth]{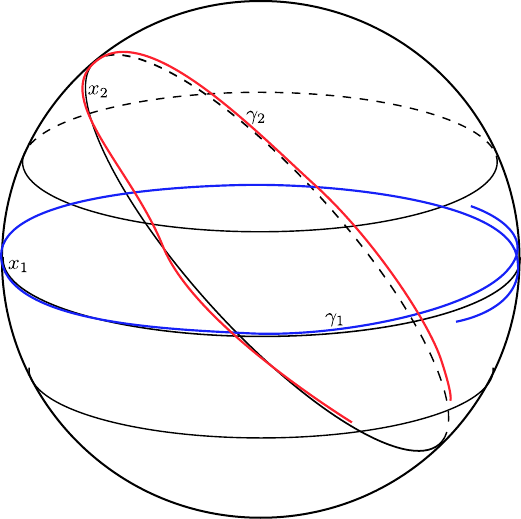}
    \caption{Picture for Corollary \ref{cor:circle}}
    \label{fig:1}
\end{figure}

Finally, since $\gamma$ is a closed curve in $Q$ that is topologically a cylinder, $\gamma$ winds around $Q$ only once. By Lemma \ref{lem:42} again, we can choose $\epsilon_5$ small and $K$ large so that it is a closed curve within $\delta'' R_1$ distance to some great circle in $P$.
\end{proof}

\subsection{The proof of Theorem \ref{thm:main2}}

Now, let's prove Theorem \ref{thm:main2}:
\begin{proof}[Proof of Theorem \ref{thm:main2}]
	The statement of Theorem \ref{thm:main2} is invariant under scaling. Hence, we may assume that \eqref{eqn:main2} implies \eqref{eqn:goodinf} for $\epsilon\leq \epsilon_0$. The desired constant $\epsilon_0$ is determined as follows.
	Let $\epsilon_1$ and $R_0$ (hence $R=R_0+1$) be given by Theorem \ref{thm:3circle}. For $\delta''>0$ to be determined by $\epsilon_1$, Corollary \ref{cor:circle} gives $\epsilon_5$ and $K>0$ such that for any $R^{i}>K+1$ (for some $i\in \mathbb N$), a component $\gamma$ of $\Sigma\cap \partial B_{R^{i}}$ is $\delta'' R^{i}$-close to a great circle on $\partial B_{R^{i}}$. In particular, after scaling down by $R^{i-1}$, the curve $\gamma$ is $\delta''R$-close to a round circle of radius $R$. By Lemma \ref{lem:graph} and Proposition \ref{pro:combine}, if $\epsilon_0$ and $\delta''$ are small enough, after the same scaling, the component of $\Sigma$ containing $\gamma$ in $B_{R^{i+3}}\setminus B_{R^{i}}$ satisfies (A2) (with $\epsilon<\epsilon_1$) of Theorem \ref{thm:3circle}. In doing so, $\delta''$ is determined by $\epsilon_1$ alone and $\epsilon_0$ is required to be smaller than $\epsilon_5$ that depends on $\epsilon_1$. 

As a result of the above analysis, the intersection of $\Sigma$ with $\partial B_r$ for $r>K+1$ is a finite union of closed curves that are almost round circles and hence the number of such curves is independent of $r$. This proves the claim about the finiteness of ends.

For $i_0$ sufficiently large, $\Sigma\setminus B_{R^{i_0}}$ has finitely many connected components, each corresponding to an end. Let $\Sigma_0$ be one of them. We would like to apply Theorem \ref{thm:3circle} to $\Sigma_0$. We have shown that for $i>i_0$, $\Sigma_0\cap (B_{R^{i+4}}\setminus B_{R^{i+1}})$ satisfies (A2) of Theorem \ref{thm:3circle}. In order to apply Theorem \ref{thm:3circle}, it remains to verify \eqref{eqn:residue}. If the number of ends of $\Sigma$ is one, then $\Sigma_0$ is the only end and $\Sigma_0\cap \partial B_r$ ($r>R^{i_0}$) is the boundary of a bounded piece of $\Sigma$, which implies \eqref{eqn:residue}. If instead of the number of ends, we assume \eqref{eqn:main1}, then \eqref{eqn:residue} follows by taking the limit $r\to \infty$ in the definition of $\tau_1$ and $\tau_2$. More precisely, since the definition of $\tau_1$ and $\tau_2$ (see \eqref{eqn:tau_c} and \eqref{eqn:tau_S}) is independent of $r$, we have
    \begin{equation}
        \tau_2(\Sigma_0,S)= \lim_{r\to \infty} \int_{\Sigma_0\cap \partial B_r} - 2 (Sx\cdot \n) \pfrac{H}{\nu} + 2 H \pfrac{(Sx\cdot \n)}{\nu} + \abs{H}^2 (Sx\cdot \nu) ds. 
    \end{equation}
By \eqref{eqn:main1} and Corollary \ref{cor:circle}, $\tau_2(\Sigma_0,S)=0$. The proof for $\tau_1$ is similar.

Theorem \ref{thm:3circle} then implies that
\[
	\int_{\Sigma_0\cap (B_{R^{i+1}}\setminus B_{R^i})} \abs{A}^2 dV_{\Sigma_0} \leq C e^{-\alpha i}
\]
for some $\alpha>0$, from which $\norm{A}_{L^2(\Sigma)}<\infty$ follows. 
\end{proof}

\appendix

\section{A three-circle property for biharmonic functions}
This section is devoted to the proof of Lemma \ref{lem:lem32}. For any function $u$ on the plane, let $(r,\theta)$ be the polar coordinates. We start with the following formula:
\begin{equation}
	\label{eqn:hessianu}
\abs{\nabla^{2}u}^{2}
=(\partial_{r}^{2}u)^{2}
+\frac{2}{r^{2}}\Bigl(\partial_{r\theta}^{2}u-\frac1r\partial_{\theta}u\Bigr)^{2}
+\frac1{r^{4}}\bigl(\partial_{\theta}^{2}u+r\partial_{r}u\bigr)^{2}.
\end{equation}
The computation for the above formula is omitted.

Using \eqref{eqn:hessianu}, we compute
\begin{eqnarray*}
	&& \int_{D_R\setminus D_r} \abs{\nabla^2 u}^2 d x \\
	&=& \int_{D_R\setminus D_r} 
\left[(\partial_{r}^{2}u)^{2} +\frac{2}{r^{2}}\Bigl(\partial_{r\theta}^{2}u-\frac1r\partial_{\theta}u\Bigr)^{2} +\frac1{r^{4}}\bigl(\partial_{\theta}^{2}u+r\partial_{r}u\bigr)^{2}\right] r\,dr\,d\theta\\
&=&\int_0^{2 \pi} \int_{\log r}^{\log R} e^{-2 t}\left[\left(\partial_t^2 u-\partial_t u\right)^2+2\left(\partial_{t \theta}^2 u-\partial_\theta u\right)^2+\left(\partial_\theta^2 u+\partial_t u\right)^2\right] d t\, d \theta.
\end{eqnarray*}
Motivated by this, we define
\[
Q_u(t):=\int_0^{2\pi} e^{-2 t}\left[\left(\partial_t^2 u-\partial_t u\right)^2+2\left(\partial_{t \theta}^2 u-\partial_\theta u\right)^2+\left(\partial_\theta^2 u+\partial_t u\right)^2\right] d\theta,
\]
so that
\[
\int_{D_R\setminus D_r} \abs{\nabla^2 u}^2 dx = \int_{\log r}^{\log R} Q_u(t)\,dt.
\]

\subsection{Reduce to linear algebra}

Recall from Section~\ref{sec:pre} that a biharmonic function $u$ on an annulus has the expansion \eqref{eqn:biexp}.  Setting $t=\log r$ turns the radial functions in \eqref{eqn:biexp} into exponentials in $t$, so biharmonic functions are linear combinations (converging in $L^2$) of
\begin{gather*}
	1, t, e^{2t}, te^{2t} , \\
	e^t \cos \theta, e^{-t} \cos \theta, e^{3t} \cos \theta, te^t \cos \theta, \text{(same for $\sin$)} \\
	e^{nt} \cos n\theta, e^{-nt}\cos n\theta, e^{(n+2)t} \cos n\theta, e^{-(n-2)t} \cos n\theta,  \text{(same for $\sin$)} .
\end{gather*}
Since the eigenfunctions are orthogonal in the $L^2$ inner product of $S^1$, it suffices to prove Lemma \ref{lem:lem32} for
\begin{align*}
 u_0 & = a + bt + c e^{2t} + d te^{2t} \\	
 u_1 & = \left( a e^t + b e^{-t} + c e^{3t} + d te^t \right) \cos \theta \\	
 u_n & = \left( a e^{nt} + b e^{-nt} + c e^{(n+2)t} + d e^{-(n-2)t} \right) \cos n\theta.
\end{align*}
Note that the assumption that $D_1$ and $\tilde{D}_1$ vanish in Lemma \ref{lem:lem32} amounts to the claim that $d=0$ in $u_1$.
If we write $Q_*$ for $Q_{u_*}$, we summarize some direct computations in the form of a lemma:
\begin{lem}
	\label{lem:Q}
	\begin{align*}
		Q_{u_0}(t) &=2\pi e^{-2t}\Big[\big(-b+(2c+3d)e^{2t}+2dte^{2t}\big)^2+\big(b+(2c+d)e^{2t}+2dte^{2t}\big)^2\Big]\\
		Q_{u_1}(t)&=16\pi b^2e^{-4t}-8\pi bde^{-2t}+4\pi d^2+24\pi cde^{2t}+48\pi c^2e^{4t}\\
		Q_{u_n}(t)&=4\pi a^2n^2(n-1)^2e^{2(n-1)t}+4\pi b^2n^2(n+1)^2e^{-2(n+1)t}+4\pi c^2(n+1)^2(n^2+2)e^{2(n+1)t} \\
			  & +4\pi d^2(n-1)^2(n^2+2)e^{2(1-n)t}+8\pi ac\,n^2(n^2-1)e^{2nt} \\
			  & +8\pi bd\,n^2(n^2-1)e^{-2nt}-16\pi cd\,(n^2-1)e^{2t}.
	\end{align*}
\end{lem}

The proof of Lemma~\ref{lem:lem32} is then reduced to the claim that there exists $L_0>0$ such that for any $L>L_0$, we have
\begin{equation}\label{eqn:reduction}
Q_*(0) \le \frac{1}{4} \bigl( Q_*(L) + Q_*(-L) \bigr).
\end{equation}
In fact, by translation, the above inequality implies
$$
Q_*(t) \le \frac{1}{4} \bigl( Q_*(t+L) + Q_*(t-L) \bigr)
$$
for any $t$. Integrating the above over $[2L,3L]$ and setting $R_0=e^{L_0}$ completes the proof of Lemma \ref{lem:lem32}.

Evaluating \eqref{eqn:reduction} for $u_0$, $u_1$ and $u_n$ yields three inequalities, which we collect in the following lemma.  The algebraic computations used in the proof of Lemma~\ref{lem:claims} are omitted.

\begin{lem}\label{lem:claims}
The proof of Lemma~\ref{lem:lem32} reduces to verifying the following three inequalities.

\smallskip\noindent\textup{(i)} For all real $b,c,d$ and all sufficiently large $L>0$,
\begin{equation}\label{eq:claim0}
\begin{aligned}
&(\cosh 2L-2)\bigl(b^{2}+4c^{2}+8cd+5d^{2}\bigr)+2bd\\
&+8Ld(c+d)\sinh 2L+4L^{2}d^{2}\cosh 2L\ge 0.
\end{aligned}
\end{equation}

\smallskip\noindent\textup{(ii)} For all real $b,c$ and all $L>0$,
\begin{equation}\label{eq:claim1}
8\pi\bigl(b^2+3c^2\bigr)(\cosh 4L-2)\ge 0.
\end{equation}

\smallskip\noindent\textup{(iii)} For every integer $n\ge 2$, all real $a,b,c,d$ and all sufficiently large $L>0$,
\begin{equation}\label{eq:claimn}
\begin{aligned}
&(n-1)^2\bigl(a^2n^2+d^2(n^2+2)\bigr)\bigl(\cosh 2(n-1)L-2\bigr)\\
&+(n+1)^2\bigl(b^2n^2+c^2(n^2+2)\bigr)\bigl(\cosh 2(n+1)L-2\bigr)\\
&+2n^2(n^2-1)(ac+bd)\bigl(\cosh 2nL-2\bigr)\\
&-4(n^2-1)cd\bigl(\cosh 2L-2\bigr)\ge 0.
\end{aligned}
\end{equation}
\end{lem}

To finish the proof of Lemma \ref{lem:lem32}. It suffices to prove the three inequalities above. While \eqref{eq:claim1} is trivial, the proofs for \eqref{eq:claim0} and \eqref{eq:claimn} are quite involved. They are presented in the next two subsections.

\subsection{Proof for \eqref{eq:claim0}}

\begin{proof}
Set
\[
x = \cosh 2L,\qquad s = \sinh 2L,\qquad A = x-2 = \cosh 2L-2.
\]
For \(L>\frac12\operatorname{arccosh}2\) we have \(x>2\) and therefore \(A>0\).  
By Lemma~\ref{lem:claims}\,(i) we must show
\[
E := (\cosh 2L-2)(b^{2}+4c^{2}+8cd+5d^{2}) + 2bd + 8Ld(c+d)\sinh 2L + 4L^{2}d^{2}\cosh 2L \ge 0.
\]
Substituting $x$, $s$ and $A$, the left‑hand side becomes a quadratic form in \(b,c,d\):
\[
\begin{aligned}
E &= A(b^{2}+4c^{2}+8cd+5d^{2}) + 2bd + 8Ld(c+d)s + 4L^{2}d^{2}x\\[2pt]
  &= A b^{2} + 2bd + 4A c^{2} + 8(A+Ls)cd + \bigl(5A + 8Ls + 4L^{2}x\bigr)d^{2}.
\end{aligned}
\]

We complete the square in \(b\) and \(c\).  The terms containing $b$ are $A b^{2}+2bd$:
\[
A b^{2} + 2bd = A\Bigl(b^{2} + \frac{2d}{A}b\Bigr)
= A\Bigl(b+\frac{d}{A}\Bigr)^{2} - \frac{d^{2}}{A}.
\]

The terms containing $c$ are $4A c^{2} + 8(A+Ls)cd$:
\[
\begin{aligned}
4A c^{2} + 8(A+Ls)cd
&= 4A\Bigl(c^{2} + \frac{2(A+Ls)}{A}\,cd\Bigr)\\[2pt]
&= 4A\Bigl(c+\frac{A+Ls}{A}d\Bigr)^{2} - 4\frac{(A+Ls)^{2}}{A}d^{2}.
\end{aligned}
\]

Substituting these back,
\[
\begin{aligned}
E &= A\Bigl(b+\frac{d}{A}\Bigr)^{2} + 4A\Bigl(c+\frac{A+Ls}{A}d\Bigr)^{2} \\
  &\quad + \Bigl[5A + 8Ls + 4L^{2}x - \frac{1}{A} - \frac{4(A+Ls)^{2}}{A}\Bigr]d^{2}.
\end{aligned}
\]

Now simplify the coefficient of \(d^{2}\).  Denote it by $C$:
\[
\begin{aligned}
C &= 5A + 8Ls + 4L^{2}x - \frac{1}{A} - \frac{4(A^{2}+2ALs+L^{2}s^{2})}{A}\\[2pt]
  &= 5A + 8Ls + 4L^{2}x - \frac{1}{A} - 4A - 8Ls - \frac{4L^{2}s^{2}}{A}\\[2pt]
  &= A - \frac{1}{A} + 4L^{2}\Bigl(x - \frac{s^{2}}{A}\Bigr).
\end{aligned}
\]

Using \(s^{2} = \cosh^{2}2L - 1 = x^{2}-1\) and \(A=x-2\),
\[
x - \frac{s^{2}}{A} = x - \frac{x^{2}-1}{x-2}
= \frac{x(x-2)-(x^{2}-1)}{x-2}
= \frac{1-2x}{x-2}.
\]

Hence
\[
\begin{aligned}
C &= A - \frac{1}{A} + 4L^{2}\cdot\frac{1-2x}{x-2}\\[2pt]
  &= \frac{A^{2}-1}{A} + \frac{4L^{2}(1-2x)}{A}\\[2pt]
  &= \frac{A^{2}-1 + 4L^{2}(1-2x)}{A}.
\end{aligned}
\]

Now substitute \(A = x-2\) and \(A^{2} = (x-2)^{2} = x^{2}-4x+4\):
\[
\begin{aligned}
C &= \frac{(x^{2}-4x+4)-1 + 4L^{2} - 8L^{2}x}{x-2}\\[2pt]
  &= \frac{x^{2} - 4x + 3 + 4L^{2} - 8L^{2}x}{x-2}\\[2pt]
  &= \frac{x^{2} - 4(1+2L^{2})x + (4L^{2}+3)}{x-2}.
\end{aligned}
\]

Define the numerator
\[
N(L) := x^{2} - 4(1+2L^{2})x + (4L^{2}+3), \qquad x=\cosh 2L.
\]
Because \(A = x-2 > 0\) for the $L$ we consider, the denominator $x-2$ is positive.  Thus $\operatorname{sgn}(C)=\operatorname{sgn}\bigl(N(L)\bigr)$.

Now examine \(N(L)\) for large \(L\).  Since \(\cosh 2L \sim \frac12 e^{2L}\to\infty\), the leading term \(x^{2}\) dominates:
\[
\lim_{L\to\infty} \frac{N(L)}{x^{2}}
= \lim_{L\to\infty}\Bigl(1 - \frac{4(1+2L^{2})}{x} + \frac{4L^{2}+3}{x^{2}}\Bigr) = 1 > 0.
\]
By continuity, there exists \(L_{0}>0\) (independent of $b,c,d$) such that \(N(L)\ge 0\) for all \(L\ge L_{0}\).  For any such \(L\) we have \(A>0\) and \(C\ge 0\).

Therefore, when \(L\ge L_{0}\), the expression $E$ takes the form
\[
E = A\Bigl(b+\frac{d}{A}\Bigr)^{2} + 4A\Bigl(c+\frac{A+Ls}{A}d\Bigr)^{2} + C\,d^{2},
\]
with $A>0$, $4A>0$ and $C\ge 0$.  Each term is a non‑negative multiple of a square, so $E\ge 0$ for all real $b,c,d$ (the coefficient $a$ does not appear, so the statement holds for any $a$).  This completes the proof.
\end{proof}

\subsection{Proof for \eqref{eq:claimn}}

\begin{pro}
Let $n\ge 2$ be an integer and let $L>0$.  Define
\begin{align*}
X&=\cosh 2(n-1)L-2, & Y&=\cosh 2(n+1)L-2,\\[2pt]
Z&=\cosh 2nL-2, & W&=\cosh 2L-2.
\end{align*}
Then for all $a,b,c,d\in\mathbb{R}$,
\begin{multline}\label{eq:main}
(n-1)^2\bigl(a^2n^2+d^2(n^2+2)\bigr)X
+(n+1)^2\bigl(b^2n^2+c^2(n^2+2)\bigr)Y\\
+2n^2(n^2-1)(ac+bd)Z
-4(n^2-1)cdW \ge 0.
\end{multline}
\end{pro}

\begin{proof}
The left-hand side of \eqref{eq:main} is a quadratic form in $\mathbf v=(a,b,c,d)$.
We diagonalize it by completing the square in three steps.

\medskip\noindent\textbf{Step 1: the $(a,c)$--block.}
The terms involving $a$ and $c$ are picked out from \eqref{eq:main}:
\[
\begin{aligned}
&\text{from $a^2$:}\quad (n-1)^2n^2X\,a^2,\\[2pt]
&\text{from $ac$:}\quad 2n^2(n^2-1)Z\,ac,\\[2pt]
&\text{from $c^2$:}\quad (n+1)^2(n^2+2)Y\,c^2.
\end{aligned}
\]
Set
\[
A:=(n-1)^2n^2X,\qquad
B:=n^2(n^2-1)Z,\qquad
C:=(n+1)^2(n^2+2)Y,
\]
so that the $(a,c)$--block reads $A a^2 + 2B ac + C c^2$.
Completing the square in $a$ gives
\[
A a^2 + 2B ac + C c^2
= A\Bigl(a + \frac{B}{A}c\Bigr)^2 + \Bigl(C - \frac{B^{2}}{A}\Bigr)c^{2}.
\]

Compute $B^{2}/A$:
\[
\frac{B^{2}}{A}
= \frac{n^{4}(n^{2}-1)^{2}Z^{2}}{(n-1)^{2}n^{2}X}
= \frac{n^{2}(n+1)^{2}Z^{2}}{X}.
\]

The remaining coefficient of $c^{2}$ is
\[
R_{c} := C - \frac{B^{2}}{A}
= (n+1)^{2}\Bigl((n^{2}+2)Y - \frac{n^{2}Z^{2}}{X}\Bigr).
\]
Set
\[
D=(n^2+2)XY-n^2Z^2.
\]
Then $(n^{2}+2)Y - n^{2}Z^{2}/X = D/X$, so
\[
R_{c} = \frac{(n+1)^{2}D}{X}.
\]

\medskip\noindent\textbf{Step 2: absorbing the $cd$--term.}
The terms containing $c$ are now $R_{c} c^{2}$, and \eqref{eq:main} also contains
the cross term $-4(n^{2}-1)W\,cd$ and the $d^{2}$ term $(n-1)^{2}(n^{2}+2)X\,d^{2}$.
Combine the $c^{2}$ and $cd$ terms by completing the square in $c$:
\[
R_{c}c^{2} - 4(n^{2}-1)W\,cd
= R_{c}\Bigl(c - \frac{2(n^{2}-1)W}{R_{c}}d\Bigr)^{2}
  - \frac{4(n^{2}-1)^{2}W^{2}}{R_{c}}\,d^{2}.
\]

The ``cost'' in $d^{2}$ simplifies as
\[
\frac{4(n^{2}-1)^{2}W^{2}}{R_{c}}
= \frac{4(n^{2}-1)^{2}W^{2}}{\frac{(n+1)^{2}D}{X}}
= \frac{4(n-1)^{2}W^{2}X}{D}.
\]

\medskip\noindent\textbf{Step 3: the $(b,d)$--block.}
After the first two steps, the remaining $d^{2}$ coefficient is the original one
minus the cost from Step~2:
\[
\begin{aligned}
R_{d} &:= (n-1)^{2}(n^{2}+2)X - \frac{4(n-1)^{2}W^{2}X}{D}\\[2pt]
      &= (n-1)^{2}X\Bigl((n^{2}+2) - \frac{4W^{2}}{D}\Bigr).
\end{aligned}
\]

Now we complete the square in $b$.  The terms involving $b$ are the $b^{2}$ term
$R_{b}\,b^{2} := (n+1)^{2}n^{2}Y\,b^{2}$ together with the cross term
$2n^{2}(n^{2}-1)Z\,bd$:
Since $R_b=(n+1)^2n^2Y$, completing the square gives
\[
R_{b} b^{2} + 2n^{2}(n^{2}-1)Z\,bd
= R_{b}\Bigl(b + \frac{(n-1)Z}{(n+1)Y}d\Bigr)^{2}
  - \frac{n^{2}(n-1)^{2}Z^{2}}{Y}\,d^{2}.
\]

Thus the $(b,d)$--block contributes
\[
R_{b}\Bigl(b + \frac{(n-1)Z}{(n+1)Y}d\Bigr)^{2}
+ \Bigl(R_{d} - \frac{n^{2}(n-1)^{2}Z^{2}}{Y}\Bigr)d^{2}.
\]

Compute the final $d^{2}$ coefficient:
\[
\begin{aligned}
R_{d} - \frac{n^{2}(n-1)^{2}Z^{2}}{Y}
&= (n-1)^{2}X\Bigl((n^{2}+2) - \frac{4W^{2}}{D}\Bigr) - \frac{n^{2}(n-1)^{2}Z^{2}}{Y}\\[2pt]
&= (n-1)^{2}\Bigl[X(n^{2}+2) - \frac{n^{2}Z^{2}}{Y}\Bigr] - (n-1)^{2}\frac{4XW^{2}}{D}.
\end{aligned}
\]

Note that $X(n^{2}+2) - n^{2}Z^{2}/Y = \frac{(n^{2}+2)XY - n^{2}Z^{2}}{Y} = D/Y$.  Hence
\[
R_{d} - \frac{n^{2}(n-1)^{2}Z^{2}}{Y}
= (n-1)^{2}\Bigl(\frac{D}{Y} - \frac{4XW^{2}}{D}\Bigr)
= (n-1)^{2}\,\frac{D^{2} - 4XYW^{2}}{DY}.
\]

\medskip\noindent\textbf{Conclusion.}
Collecting the three perfect squares and the final residual, the left-hand side of
\eqref{eq:main} equals
\begin{multline*}
(n-1)^{2}n^{2}X\Bigl(a+\frac{(n+1)Z}{(n-1)X}c\Bigr)^{2}
+\frac{(n+1)^{2}D}{X}\Bigl(c-\frac{2(n-1)WX}{(n+1)D}d\Bigr)^{2}\\
+(n+1)^{2}n^{2}Y\Bigl(b+\frac{(n-1)Z}{(n+1)Y}d\Bigr)^{2}
+(n-1)^{2}\frac{D^{2}-4XYW^{2}}{DY}\,d^{2}.
\end{multline*}
By Lemma~\ref{lem:Dpos} and Lemma~\ref{lem:D2} below, there exists a
universal constant $L_{0}>0$ such that for every $n\ge 2$ we have $D>0$ and
$D^{2}\ge 4XYW^{2}$ whenever $L\ge L_{0}$.  Hence every term in the above sum is
non-negative (each coefficient in front of a square is positive), and
\eqref{eq:main} follows.
\end{proof}

\begin{lem}\label{lem:Dpos}
There exists a universal constant $L_1>0$ such that for every integer $n\ge 2$
and every $L\ge L_1$,
\[
D=(n^2+2)XY-n^2Z^2>0.
\]
\end{lem}

\begin{proof}
Set $q=e^{-2L}\in(0,1)$.  Recall the definitions
\[
X=\cosh 2(n-1)L-2,\quad
Y=\cosh 2(n+1)L-2,\quad
Z=\cosh 2nL-2.
\]
Using $\cosh 2kL = \tfrac12(q^{-k}+q^{k})$ and $\cosh 2kL-1 = \tfrac12(q^{-k/2}-q^{k/2})^{2}$,
we obtain
\[
X=\frac{1}{2q^{n-1}}(1-4q^{n-1}+q^{2n-2}),\qquad
Y=\frac{1}{2q^{n+1}}(1-4q^{n+1}+q^{2n+2}),\qquad
Z=\frac{1}{2q^{n}}(1-4q^{n}+q^{2n}).
\]
Define the auxiliary polynomials $\alpha_k(q)=1-4q^{k}+q^{2k}$, so that
\[
X=\frac{\alpha_{n-1}(q)}{2q^{n-1}},\qquad
Y=\frac{\alpha_{n+1}(q)}{2q^{n+1}},\qquad
Z=\frac{\alpha_{n}(q)}{2q^{n}}.
\]

Now compute $D=(n^{2}+2)XY-n^{2}Z^{2}$:
\[
\begin{aligned}
D &= (n^{2}+2)\frac{\alpha_{n-1}(q)}{2q^{n-1}}\cdot\frac{\alpha_{n+1}(q)}{2q^{n+1}}
    - n^{2}\frac{\alpha_{n}(q)^{2}}{4q^{2n}}\\[4pt]
  &= \frac{1}{4q^{2n}}\Bigl[
     (n^{2}+2)\,\alpha_{n-1}(q)\,\alpha_{n+1}(q)
     - n^{2}\,\alpha_{n}(q)^{2}
     \Bigr].
\end{aligned}
\]

Define
\[
G_{n}(q):=(n^{2}+2)(1-4q^{n-1}+q^{2n-2})(1-4q^{n+1}+q^{2n+2})
-n^{2}(1-4q^{n}+q^{2n})^{2},
\]
which is precisely the bracket above.  Then $D=\frac{q^{-2n}}{4}G_{n}(q)$.  Since
$q^{-2n}/4>0$, the sign of $D$ is the same as the sign of $G_{n}(q)$.

We now verify $G_{n}(q)>0$ for all $n\ge 2$ and sufficiently small $q$.
Expanding the product defining $G_{n}(q)$, the constant term equals
$(n^{2}+2)-n^{2}=2$.  Every other term contains a positive power of $q$, and
the smallest such power is $q^{n-1}$.  Moreover, the coefficient of each
non-constant term is bounded in absolute value by $C_{0}(n^{2}+2)$ for some
universal constant $C_{0}$.  Since $0<q<1$ implies $q^{k}\le q^{n-1}$ for every
$k\ge n-1$, the sum of the absolute values of all non-constant terms is at most
$C(n^{2}+2)q^{n-1}$ for another universal constant $C$.  Consequently
\[
G_{n}(q)\ge 2-C(n^{2}+2)q^{n-1}.
\]

For $0<q\le\tfrac1{12}$ the sequence $(n^{2}+2)q^{n-1}$ is decreasing in
$n\ge 2$ (because $\frac{(n+1)^{2}+2}{n^{2}+2}\,q\le\frac{11}{6}\cdot
\frac1{12}<1$), hence its maximum over $n\ge 2$ is attained at $n=2$, giving
$(n^{2}+2)q^{n-1}\le 6q$.  Therefore
\[
G_{n}(q)\ge 2-6Cq\qquad\text{for all }n\ge 2.
\]
Choosing $q$ sufficiently small guarantees $G_{n}(q)>0$ for every $n\ge 2$.
Hence there exists a universal constant $L_{1}>0$ such that $G_{n}(q)>0$
whenever $L\ge L_{1}$, and the proof is complete.
\end{proof}

\begin{lem}\label{lem:D2}
There exists a universal constant $L_2>0$ such that for every integer $n\ge 2$
and every $L\ge L_2$,
\[
D^2\ge 4XYW^2.
\]
\end{lem}

\begin{proof}
Set $q=e^{-2L}$ and define $\alpha_k(q)=1-4q^{k}+q^{2k}$ and
$\beta(q)=1-4q+q^{2}$ as before.  Recall from the proof of
Lemma~\ref{lem:Dpos} that
\[
X=\frac{\alpha_{n-1}(q)}{2q^{n-1}},\quad
Y=\frac{\alpha_{n+1}(q)}{2q^{n+1}},\quad
Z=\frac{\alpha_n(q)}{2q^{n}},\quad
W=\cosh 2L-2=\frac{\beta(q)}{2q}.
\]

We also have $D=\frac{q^{-2n}}{4}G_n(q)$, hence
\[
16q^{4n}D^{2}=G_n(q)^{2}.
\]

For the product $XYW^{2}$ we compute
\[
\begin{aligned}
16q^{4n}(4XYW^{2})
&= 64\,q^{4n}\cdot
\frac{\alpha_{n-1}(q)}{2q^{n-1}}\cdot
\frac{\alpha_{n+1}(q)}{2q^{n+1}}\cdot
\frac{\beta(q)^{2}}{4q^{2}} \\[4pt]
&= 64\,q^{4n}\cdot
\frac{\alpha_{n-1}(q)\,\alpha_{n+1}(q)\,\beta(q)^{2}}{16\,q^{2n+2}} \\[4pt]
&= 4q^{2n-2}\,\alpha_{n-1}(q)\,\alpha_{n+1}(q)\,\beta(q)^{2}.
\end{aligned}
\]

Therefore
\[
H_{n}(q):=16q^{4n}(D^{2}-4XYW^{2})
=G_{n}(q)^{2} - 4q^{2n-2}\,\alpha_{n-1}(q)\,\alpha_{n+1}(q)\,\beta(q)^{2}.
\]

Now bound the subtracted term.  For $0<q\le\tfrac1{12}$, each $\alpha_k(q)$
and $\beta(q)$ lies in $(0,1]$.  Hence
\[
4q^{2n-2}\,\alpha_{n-1}(q)\,\alpha_{n+1}(q)\,\beta(q)^{2}
\le 4q^{2n-2}
\le 4q^{2}
\le \frac1{36}.
\]

On the other hand, Lemma~\ref{lem:Dpos} shows that $G_{n}(q)\to 2$ uniformly
as $q\to 0$, so there exists $q_{0}>0$ such that $G_{n}(q)>\tfrac12$ for all
$n\ge 2$ and all $0<q<q_{0}$.  Consequently
\[
H_{n}(q) > \frac14-\frac1{36}=\frac{2}{9}>0.
\]

Thus $H_{n}(q)>0$, which is equivalent to $D^{2}>4XYW^{2}$.  Taking
$L_{2}>0$ large enough (so that $q=e^{-2L}<q_{0}$) completes the proof.
\end{proof}

\bibliographystyle{alpha}
\bibliography{ref}

@article {Chen-Lamm,
    AUTHOR = {Chen, Jingyi and Lamm, Tobias},
     TITLE = {A {B}ernstein type theorem for entire {W}illmore graphs},
   JOURNAL = {J. Geom. Anal.},
  FJOURNAL = {Journal of Geometric Analysis},
    VOLUME = {23},
      YEAR = {2013},
    NUMBER = {1},
     PAGES = {456--469},
      ISSN = {1050-6926,1559-002X},
   MRCLASS = {53A05 (35J61 53C42)},
  MRNUMBER = {3010288},
MRREVIEWER = {Andreas\ Gastel},
       DOI = {10.1007/s12220-011-9264-2},
       URL = {https://doi.org/10.1007/s12220-011-9264-2},
}

@article {ka,
    AUTHOR = {Kapouleas, Nikolaos},
     TITLE = {Complete embedded minimal surfaces of finite total curvature},
   JOURNAL = {J. Differential Geom.},
  FJOURNAL = {Journal of Differential Geometry},
    VOLUME = {47},
      YEAR = {1997},
    NUMBER = {1},
     PAGES = {95--169},
      ISSN = {0022-040X,1945-743X},
   MRCLASS = {53A10},
  MRNUMBER = {1601434},
MRREVIEWER = {Rabah\ Souam},
       URL = {http://projecteuclid.org/euclid.jdg/1214460038},
}

@article {Kuwert-Schatzer,
    AUTHOR = {Kuwert, Ernst and Sch\"atzle, Reiner},
     TITLE = {The {W}illmore flow with small initial energy},
   JOURNAL = {J. Differential Geom.},
  FJOURNAL = {Journal of Differential Geometry},
    VOLUME = {57},
      YEAR = {2001},
    NUMBER = {3},
     PAGES = {409--441},
      ISSN = {0022-040X,1945-743X},
   MRCLASS = {53C44},
  MRNUMBER = {1882663},
MRREVIEWER = {Shu-Cheng\ Chang},
       URL = {http://projecteuclid.org/euclid.jdg/1090348128},
}

@article {Langer,
    AUTHOR = {Langer, Joel},
     TITLE = {A compactness theorem for surfaces with {$L_p$}-bounded second
              fundamental form},
   JOURNAL = {Math. Ann.},
  FJOURNAL = {Mathematische Annalen},
    VOLUME = {270},
      YEAR = {1985},
    NUMBER = {2},
     PAGES = {223--234},
      ISSN = {0025-5831,1432-1807},
   MRCLASS = {53C42 (58G30)},
  MRNUMBER = {771980},
MRREVIEWER = {Renato\ Tribuzy},
       DOI = {10.1007/BF01456183},
       URL = {https://doi.org/10.1007/BF01456183},
}

@article {Li,
    AUTHOR = {Li, Yuxiang},
     TITLE = {Some remarks on {W}illmore surfaces embedded in {$\Bbb{R}^3$}},
   JOURNAL = {J. Geom. Anal.},
  FJOURNAL = {Journal of Geometric Analysis},
    VOLUME = {26},
      YEAR = {2016},
    NUMBER = {3},
     PAGES = {2411--2424},
      ISSN = {1050-6926,1559-002X},
   MRCLASS = {53A05 (53A30)},
  MRNUMBER = {3511481},
MRREVIEWER = {Changxiong\ Nie},
       DOI = {10.1007/s12220-015-9631-5},
       URL = {https://doi.org/10.1007/s12220-015-9631-5},
}

@article {Li-Yin,
    AUTHOR = {Li, Yuxiang and Yin, Hao},
     TITLE = {3-circle theorem for {W}illmore surface {I}},
   JOURNAL = {J. Lond. Math. Soc. (2)},
  FJOURNAL = {Journal of the London Mathematical Society. Second Series},
    VOLUME = {111},
      YEAR = {2025},
    NUMBER = {5},
     PAGES = {Paper No. e70165, 28},
      ISSN = {0024-6107,1469-7750},
   MRCLASS = {53A05 (53C21)},
  MRNUMBER = {4899930},
       DOI = {10.1112/jlms.70165},
       URL = {https://doi.org/10.1112/jlms.70165},
}

@article {Li-Yin-Zhou,
      title={3-circle Theorem for Willmore surfaces II--degeneration of the complex structure}, 
      author={Yuxiang Li and Hao Yin and Jie Zhou},
      journal = {arXiv preprint https://arxiv.org/abs/2411.06453},
      year={2024},
      eprint={2411.06453},
      archivePrefix={arXiv},
      primaryClass={math.DG},
      url={https://arxiv.org/abs/2411.06453}, 
}

@article {Meeks-Perez-Ros,
    AUTHOR = {Meeks, III, William H. and P\'erez, Joaqu\'in and Ros,
              Antonio},
     TITLE = {Local removable singularity theorems for minimal laminations},
   JOURNAL = {J. Differential Geom.},
  FJOURNAL = {Journal of Differential Geometry},
    VOLUME = {103},
      YEAR = {2016},
    NUMBER = {2},
     PAGES = {319--362},
      ISSN = {0022-040X,1945-743X},
   MRCLASS = {53A10 (53C42)},
  MRNUMBER = {3504952},
MRREVIEWER = {Giuseppe\ Tinaglia},
       URL = {http://projecteuclid.org/euclid.jdg/1463404121},
}

\end{document}